\documentclass[a4paper]{amsart}
\usepackage[utf8]{inputenc}
\usepackage[T1]{fontenc}
\usepackage[british]{babel}
\usepackage[all,cmtip]{xy}
\usepackage{tikz}
\usepackage{tikz-cd} 
\usepackage{xcolor}
\usepackage{amsmath,amsfonts,amssymb}
\usepackage{mathabx}
\usepackage{verbatim}
\usepackage[abbrev]{amsrefs}
\usepackage{amsthm}
\usepackage{graphicx}
\usepackage{stmaryrd}
\usepackage{mathbbol}
\usepackage{quiver}

\newcommand{\N}{\mathbb{N}}

\newcommand{\R}{\mathbb{R}}

\newcommand{\Part}{\mathcal{P}}

\newcommand{\set}[1]{\left\{#1\right\}}
\newcommand{\Id}{\text{Id}}
\newtheorem{theorem}{Theorem}[section]
\newtheorem*{theorem*}{Goal}

\newtheorem{lemma}[theorem]{Lemma}
\newtheorem{corollary}[theorem]{Corollary}

\theoremstyle{definition}
\newtheorem{exemple}[theorem]{Example}
\newtheorem{exemples}[theorem]{Examples}

\newtheorem{definition}[theorem]{Definition}
\theoremstyle{remark}
\newtheorem{remark}[theorem]{Remark}
\newtheorem{remarks}[theorem]{Remarks}

\newenvironment{tfae}
{
\begin{enumerate}}
{\end{enumerate}}

\allowdisplaybreaks[4]

\title[A categorical view on the converse Lyapunov theorem]{A categorical view on\\ the converse Lyapunov theorem}
\author{Sébastien Mattenet, Raphaël Jungers }

\address[Mattenet Sébastien]{Mathematical engineering department, Av.\ Georges Lemaitre 4, 1348 Ottignies-Louvain-la-Neuve, Belgium}

\email{sebastien.mattenet@uclouvain.be}

\begin{document}

\begin{abstract}
In 1892, Lyapunov provided a fundamental contribution to stability theory by introducing so-called \emph{Lyapunov functions} and \emph{Lyapunov equilibria}. He subsequently showed that, for linear systems, the two concepts are equivalent. These concepts have since been extended to diverse types of dynamical systems, and in all settings the equivalence remains valid. However, this involves an often technical proof in each new setting where the concepts are introduced. In this article, we investigate a categorical framework where these results can be unified, exposing a single underlying reason for the equivalence to hold in all cases. First we  define what is a dynamical system. Then we introduce the notion of a level-set morphism, which in turn allows us to define the concepts of a Lyapunov equilibrium and a Lyapunov function in a categorical setting. We conclude by a proof of their equivalence.
\end{abstract}



\maketitle

\section{Introduction}
Lyapunov theory is a cornerstone of the study of dynamical systems. Lyapunov's realisation in his Ph.D.\ thesis~\cite{Lyapunov} was that one could infer stability of a physical system without knowing the true energy function. Instead, it is enough to find a so-called Lyapunov function. Surprisingly, finding a Lyapunov function is also a necessary condition for having a (Lyapunov) equilibrium. 
The notion of Lyapunov equilibrium (and Lyapunov stability) has since then been expanded to many different systems, with always as the key result that having a Lyapunov function (or lack thereof) is equivalent to proving  an equilibrium (or lack thereof).
See \cite{bacciotti2005liapunov}*{Theorem 2.5} or \cite{khalil_nonlinear_2002}*{Theorem 4.17} for classical Lyapunov theorems. 
While one quickly develops an intuition for the translation between Lyapunov equilibria and functions, the proofs themselves are quite technical and seem to require a machinery too complex and too specific, hindering the view of a general proof holding for all types of dynamical systems at once. (To illustrate this point: the proof of \cite{khalil_nonlinear_2002}*{Theorem 4.17} is six pages long.) 
The goal of this paper is 1) to write in categorical language the definitions of equilibrium and Lyapunov function, 2) to make precise the equivalence between the two notions and 3) to formalise how change, codified in one setting, translates to change, codified in the other.

Related work includes \cite{ames2006categorical}, which provides another categorical definition of dynamical system, but does not study Lyapunov theory.

\section{Dynamical systems}
Dynamical systems are systems that evolve over time. A dynamical system is modelled as an action of time on some set of states. Such a system is \textbf{deterministic} when repeating the same action from a given state will always yield the same resulting state. A dynamical system is \textbf{closed} if through evolution of time one remains inside of the state-space.
Let us formally define a dynamical system as an action of a monoid on a metric space. Let $M$ be a space and $T$ a monoid, then a \emph{deterministic closed} dynamical system is an action of $T$ on $M$, that is a functor $F\colon T\to End(M)$.
\begin{exemples}
\begin{enumerate}
    \item Let $M=T=\R$ and pick a $v_0\in \R$. Let $F_t(x)=x+v_0\cdot t$; this is a \textbf{uniform linear movement}.
    \item Let $M=\R^n$, and $T=(\N,+)$. Let $A\in \R^{n\times n}$ be a matrix. Let the evolution functor be defined by $F_1 \colon M\to M \colon  x\in M\mapsto Ax$. ($F_n=\left.F_1\right.^n$ is imposed by composition in the monoid). This is a \textbf{linear discrete time system}.
    \item Let $M=\R^n$ and $T=\set{1,\dots,n}^*$ the set of words on an alphabet of $n$ letters; the operation on $T$ is the concatenation of words. Let $A_i\in \R^{n\times n},\quad i=1,\dots, n$ be matrices. Let the evolution functor be defined by
	\[
		F_i \colon M\to M \colon  x\mapsto A_ix
		\]
		and composition. This is a \textbf{switching system}.
    \item Let $M=\R^n$, and $T=\set{\perp,\top}^*$ with concatenation. Let $A\in \R^{n\times n}$ be a matrix and $b\in \R^n$ be a vector. Let the evolution functor be defined by $F_\perp \colon M\to M \colon  x\mapsto Ax$ and $F_\top\colon M\to M \colon  x\mapsto Ax+b$. $F$ is defined on any other word by composition in the monoid. This is a \textbf{linear discrete time system with discrete control} $b$.
    \item Let $M=S^n$ be the n-sphere and $T=O(n+1)$ the group of distance preserving transformation on $\R^{n+1}$. The action is given by applying the transformation to the embedded sphere.
\end{enumerate}
\end{exemples}
One can extend the definition to also include \textit{open} systems and \textit{stochastic} systems, see for example \cite{smithe2022open}. Other definitions based on derivations are possible; see for example \cite{ames2006stability}. We chose our definition because we are interested in modelling Lyapunov's second method for stability, extracting the categorically relevant concepts.
\section{Invariants}
When considering the evolution of a system, one could consider invariants of the system. By definition an invariant is a quantity $I\colon M\to X$ which remains the same after the passage of time, therefore it forms a cocone over the diagram of the action. The triangle
\[\begin{tikzcd}
	M && M \\
	& X
	\arrow["{F_t}", from=1-1, to=1-3]
	\arrow["I"', from=1-1, to=2-2]
	\arrow["I", from=1-3, to=2-2]
\end{tikzcd}\]
commutes for every $t\in T$, and for $F\colon T\to End(M)$ we write $F(t)=F_t\colon M\to M$.
Knowledge about its invariants allows us to determine properties of the system, and more generally to decompose it into equivalence classes. That is, one can look at $X\slash _\sim $ where $ m\sim m' \iff I(m)=I(m')$ and deduce properties of reachability if this space does (or does not) contract to a single point. 

Now suppose one wants to define a monovariant, some quantity about the system which either only increases, or only decreases, such as entropy or energy in physical systems. For that we would need to select only one direction of time: i.e., once some time interval has passed, it is no longer possible to go backwards in time. We define 
\begin{definition}\label{def:timeline}
A timeline $(T,\leq,0,+)$ is a set $T$, equipped with a partial order $\leq$, and an operation $+$ that respects for all $t_1$, $t_2$, $t_3\in T$,
\begin{enumerate}
    \item $t_1+0=t_1=0+t_1$,
    \item $t_1+(t_2+t_3)=(t_1+t_2)+t_3$,
    \item $t_1\leq t_1+t_2$, 
    \item if $t_1\leq t_2$, then there exists a unique $t'\in T$ such that $t_2=t_1+t'$.\label{cond4}
\end{enumerate}
\end{definition}

This definition allows for some flexibility in the timeline (discrete time or not, branching time, etc.) while still permitting an unambiguous direction of time.

\begin{remarks}
	\begin{enumerate}
    \item Note that in particular, a timeline is a monoid $(T,+,0)$.
    \item From 1.\ and 3.\ it follows that $0\leq 0+t=t$ for any time $t$, therefore $0$ is the global minimum for the ordering. This means our timeline has a starting point and does not go indefinitely into the past.
\end{enumerate}
\end{remarks} 

\begin{exemples}
\begin{itemize}
    \item $T= \N$ with the usual order and sum. This is the \textbf{discrete linear timeline}.
    \item $T=\R^+$ with the usual order and sum. This is the \textbf{continuous linear timeline}.
    \item Let $A$ be a set, then $A^*$, the set of words over the alphabet $A$, also called the free monoid over $A$, can be equipped with the operation $\circ$ of concatenation, and the canonical ordering $w_1\leq w_2\iff \exists w_3\colon w_2=w_1\circ w_3$. This is a \textbf{discrete branching timeline}, used for example in arbitrary switching systems. Note that composition is not commutative, for example $T=\set{a,b}^*$, $ab+bab=abbab\neq bab+ab=babab$.
    \item Let $T= \set{f\colon [0,a]\to \R\mid a>0}$. We can define the ordering $f_1\leq f_2$ if $f_1=f_2$ on the domain of $f_1$, that is, if $f_2$ is an extension of $f_1$. Additionally one can define the concatenation operation 
   \begin{equation*}
        f_1\circ f_2=
        \begin{cases}
        f_1(x) &x\in [0,t_1] \\
        f_2(x-t_1)& x\in [t_1,t_1+t_2]
    \end{cases}
   \end{equation*}
   This is an example of a continuous branching timeline.
   \item Let T be a total order, bounded from below. We define the composition as the maximum, that is $x+y = \max(x,y)$. This is a very limited idea of timeline, so we won't use it in examples.
   \item Any ordered monoid $B$ (called a \textit{minimal positively ordered monoid} in \cite{wehrung1992injective}) is a valid timeline.
\end{itemize}
\end{exemples}

In a dynamical system $e\colon T\to End(M)$, if $T$ also is a timeline $(T,\leq,0,+)$, then we will call the system a \textbf{forward dynamical system}, to indicate it cannot go backward in time, because 0 is a global minimum.

We may now define a monovariant as a special kind of lax cocone.
\begin{definition}
Let $T$ be a timeline, $F\colon T\to M$ a forward dynamical system, $X$ a poset and $I\colon M\to X$ a function. We say $I$ is a \textbf{monovariant} if it is a lax cocone over $F$. In other words, the following diagram lax-commutes for all $t\in T$:
\[\begin{tikzcd}
	M && M \\
	& X
	\arrow["{F_t}", from=1-1, to=1-3]
	\arrow["I"', from=1-1, to=2-2,""{name=0,anchor=center}]
	\arrow["I", from=1-3, to=2-2,""{name=1, anchor=center}]
    \arrow[from=1, to=0, Rightarrow]
\end{tikzcd}\]
This means that $I(m)\leq I(F_t(m))$ for all $m\in M$.
\end{definition}

A particular example of interest will be the case when $I$ is the distance function $d(x,\cdot)$ with respect to a point $x$. When that happens, the interpretation is that the centre (the point where $I(x)=0$) is attractive.

\begin{exemples}
\begin{itemize}
    \item Let $M=\R$ evolve under discrete linear time with the following dynamic: $F(x)=\frac{x}{2}$. One can look at the distance to the origin and see that it is decreasing at each time step. 
    \item Same as before, but let $F(x)=\begin{cases}
        \frac{x}{2} & |x|\geq 2\\
        x & |x|< 2
    \end{cases}.$
    \item Let $M=\R$ evolve under discrete linear time with the dynamic $F(x)=x-1$ and consider the monovariant $I(x)=x$. One can see that it decreases at each time step.
\end{itemize}
\end{exemples}

Monovariants are useful to quickly rule out that a given region of $M$ is reachable from a given starting point. As a demonstration, in all of the above examples we know that from a starting point $x= 3$ or $x=-3$ we will never reach $10$ thanks to monovariance. 

We now define the following: 
\begin{definition}
Given a point $1\to M$ (or a subset $A\hookrightarrow M$), we call its image through $F$ the \textbf{future} of the point (or subset).
\end{definition}

\begin{definition}
We say a point $X$ (or set) is an \textbf{equilibrium} if it is its own future, that is $F(X)=X$.
\end{definition}

\begin{definition}
We define an \textbf{attractor} as a point for whom the distance is a monovariant.
\end{definition}

\begin{theorem}
    An attractor is an equilibrium.
\end{theorem}
\begin{proof}
    Let's call this point $x^*$. We measure the distance of its future to itself. By definition of attractor, $d(x^*,F_t(x^*))\leq d(x^*,x^*)=0$ which implies $d(x^*,F_t(x^*))=0$ which in turn implies $x^*=F_t(x^*)$.
\end{proof}

\begin{definition}
    Let $M$, $X$ be as before. Given two morphisms (or invariant candidates) $I$, $J\colon M\to X$, we say $I$ is a rough approximation of $J$ if there are transformations $A$, $B\colon X\to X$ such that
\[\begin{tikzcd}
	& M \\
	X && X \\
	& X
	\arrow["I"', from=1-2, to=2-1]
	\arrow[""{name=0, anchor=center, inner sep=0}, "J"{description}, from=1-2, to=3-2]
	\arrow["I", from=1-2, to=2-3]
	\arrow["A"', from=2-1, to=3-2]
	\arrow["B", from=2-3, to=3-2]
	\arrow[shorten >=5pt, Rightarrow, from=2-1, to=0]
	\arrow[shorten <=5pt, Rightarrow, from=0, to=2-3]
\end{tikzcd}\] lax-commutes.
\end{definition}

In our case a rough approximation of distance is an invariant that gives us a maximal and minimal distance based on the value of the invariant. 

\section{The level-set paradigm}
An invariant $I\colon M\to X$ creates equivalence classes of points (or regions of~$M$) that have the same value. It is therefore natural to try to understand invariants trough the classes they create.
When it comes to monovariants, the more relevant notion is the concept of a level-set:

\begin{definition}
Let $F\colon A\to B$ be a morphism and let $B$ be equipped with an ordering structure. We note $F^\leq\colon B\to \Part(A)$ the morphism that to each element of $b\in B$ assigns the subset $\set{x\in A\mid f(x)\leq b }\in \Part(A)$.
\end{definition}
Another way to write this would be that the level-set is the composition of the preimage with the downward closure operator. Since we have already chosen the convention that $\Part(f)$ is the covariant functor, we keep the definition above to avoid confusion.\footnote{Some authors follow the convention that $\Part(f)$ refers to the preimage functor, so that $\Part(f)$ is a contravariant functor.}

\begin{exemple}
    Consider the discrete time, linear system $M=\R^n$, $T=\N$, $A\in \R^{n\times n}$, $ F_1(x)=Ax$. Suppose we are looking at a monovariant of the form
	\[
		V\colon {\R^n\to \R \colon x\mapsto x^TBx}
		\]
		with $B\in \R^{n\times n}$ invertible (a quadratic Lyapunov function, as we will call them later). Let us concentrate on the case $B=\Id$. 
    The set of all points with the same value under $V$ is the surface of the sphere (an ellipse for $B\neq \Id$). If $V$ is indeed a monovariant, we can expect points of the sphere not to stay strictly in the surface, but instead on the interior volume, (the ball), as time passes. That is the relevant notion would be $\set{x\mid V(x)=r}=V^{-1}(r)$ for an invariant but $\set{x\mid V(x)\leq r}=V^{\leq}(r)$ for a monovariant.
\end{exemple}

\begin{exemple}\label{ex:semiadjunction}
    Let $F\colon T\to End(M)$ be a dynamical system. Consider a monovariant $V\colon M\to I$ where $I=[a,b]\subset \R$ is an interval. We can decompose $V$ into:
\[\begin{tikzcd}
	M & {\Part(M)} & {\Part(I)} & I
	\arrow["{\set{\cdot}}", from=1-1, to=1-2]
	\arrow["{\Part(V)}", from=1-2, to=1-3]
	\arrow[""{name=0, anchor=center, inner sep=0}, "\max", shift left=2, from=1-3, to=1-4]
	\arrow[""{name=1, anchor=center, inner sep=0}, "{\set{\cdot}}", shift left=2, from=1-4, to=1-3]
	\arrow["\dashv"{anchor=center, rotate=-90}, draw=none, from=0, to=1]
\end{tikzcd}\]
The adjunction on the right (where $\Part(I)$ and $I$ are viewed as categories) is in fact stronger, we have that $\max$ is a left inverse of $\set{\cdot}$. This is possible because an interval is a (co)complete lattice.
We can define $V_{\max}=\max\circ \Part(V)$, and we can observe that $V^{\leq} \vdash V_{\max} $. 
There is no direct adjunction between $V$ and $V^\leq$, since they don't have the same (co)domain, but the property we care about is preserved: $V$ is monovariant
\[\begin{tikzcd}
	M && M \\
	& I
	\arrow["{F_t}", from=1-1, to=1-3]
	\arrow["V"', from=1-1, to=2-2,""{name=0,anchor=center}]
	\arrow["V", from=1-3, to=2-2,""{name=1, anchor=center}]
    \arrow[from=1, to=0, Rightarrow]
\end{tikzcd}\] if and only if its level-set forms a lax cone
\[\begin{tikzcd}
	& {I} \\
	{\Part(M)} && {\Part(M)}
	\arrow["V^\leq"', from=1-2, to=2-1]
	\arrow["V^\leq", from=1-2, to=2-3,""{name=0, anchor=center, inner sep=0}]
	\arrow["{\Part(F_t)}"', from=2-1, to=2-3]
	\arrow["\alpha"', shorten >=2pt, Rightarrow, from=2-1, to=0]
\end{tikzcd}\]
if and only if $V_{\max}$ is a monovariant 
\[\begin{tikzcd}
	\Part(M) && \Part(M) \\
	& I
	\arrow["{\Part(F_t)}", from=1-1, to=1-3]
	\arrow["V_{\max}"', from=1-1, to=2-2,""{name=0,anchor=center}]
	\arrow["V_{\max}", from=1-3, to=2-2,""{name=1, anchor=center}]
    \arrow[from=1, to=0, Rightarrow]
\end{tikzcd}.\]
\end{exemple}
A particular case is $F=d(x,\cdot)$. In that case $F^\leq$ is none other than the ball function giving the ball centred on $x$. For our case it is easier to work with a restriction $B(x,\cdot)\colon \R_{>0}\to \Part(M)$, where we removed $0$ from our order. This means we can talk about being arbitrarily close to $x$ but not talk about being at $x$. 

This allows us to write the usual definition of stability:
\begin{definition}\label{def:lyapequilibrcat}
We say $x^*$ is a \textbf{Lyapunov equilibrium} for the forward dynamical system $F\colon T\to Aut(M)$ if there is a functor $\delta(\cdot)\in End(\R_{>0})$ such that the cell
\[\begin{tikzcd}[ampersand replacement=\&, column sep=small]
	{\R_{>0}} \& {\R_{>0}} \\
	{\Part(M)} \& {\Part(M)}
	\arrow["{\delta(\cdot)}", to=1-1, from=1-2,swap]
	\arrow[""{name=0, anchor=center, inner sep=0},"{B(x^*,\cdot)}"', from=1-1, to=2-1]
	\arrow["{\Part (F_{t})}"', from=2-1, to=2-2]
	\arrow[""{name=1, anchor=center, inner sep=0},"{B(x^*,\cdot)}", from=1-2, to=2-2]
	\arrow[ shorten >=3pt, Rightarrow, from=0, to=1]
\end{tikzcd}\]
admits a natural transformation for all $t\in T$.
\end{definition}
\begin{remark}
This is the classical definition of equilibrium, however since we only care about the behaviour near zero, the order $\R_{>0}$ can be replaced by $N^{op}={(N,\geq)}\simeq {(\set{\frac{1}{n}},\leq)}$. this is in keeping the classical definition of ``arbitrarily close'' as a substitute for ``at an infinitesimal distance'' which might be possible in the context of non-standard analysis.
\end{remark}
\begin{definition}
    A \textbf{Lyapunov function} for $x$ is a monovariant $V\colon M\to \R^+$ such that the restriction of $V\leq$ is a rough approximation of $B(x,\cdot)\colon \R_{>0}\to \Part(M)$, with the extra condition that $B$ be invertible. That is, the following diagrams lax-commute
\begin{equation}\label{eq:inclusion}
\begin{tikzcd}
	& {\R_{>0}} \\
	& {\R^+} \\
	{\Part(M)} && {\Part(M)}
	\arrow["{\Part(F_t)}"', from=3-1, to=3-3]
	\arrow[""{name=0, anchor=center, inner sep=0}, "{V^\leq}"', from=2-2, to=3-1]
	\arrow[""{name=1, anchor=center, inner sep=0}, "{V^\leq}", from=2-2, to=3-3]
	\arrow[hook, from=1-2, to=2-2]
	\arrow[shorten <=8pt, shorten >=8pt, Rightarrow, from=0, to=1]
\end{tikzcd}\end{equation}

\[\begin{tikzcd}[ampersand replacement=\&, column sep=small]	\& {\R_{>0}} \\
	{\R_{>0}} \&\& {\R_{>0}} \\
	\& {\Part(M)}
	\arrow[""{name=0, anchor=center, inner sep=0}, "{V^\leq}"{description}, from=1-2, to=3-2]
	\arrow["A"', from=1-2, to=2-1]
	\arrow["{B(x^*,\cdot)}"', from=2-1, to=3-2]
	\arrow["B", from=1-2, to=2-3]
	\arrow["{B(x^*,\cdot)}", from=2-3, to=3-2]
	\arrow[ shorten >=6pt, Rightarrow, from=2-1, to=0]
	\arrow[ shorten <=6pt, Rightarrow, from=0, to=2-3]
\end{tikzcd}\]
\end{definition}
\begin{remark}
    It is more common to define the Lyapunov function $V$ in the statespace and then construct it's levelset $V\leq$. Here we define it in the level-set and use the adjunction to find a function in statespace.
\end{remark}
\begin{theorem}\label{thm:equilibrium}
    A Lyapunov equilibrium is an equilibrium.
\end{theorem}
\begin{proof}
    Let $1\xrightarrow[]{x} M$ be a Lyapunov equilibrium. We can view it as the subset $\set{x}\in \Part(M)$. The future will be $\set{F_t(x)\mid t\in T}=\bigcup\Part(F_t)(\set{x})$. 
    
Separately $d(x,x)=0\leq r$ implies $x\in B(x,r)$, equivalently $\set{x}\subset B(x,r)$, for any radius $r\in \R_{>0}$.
In fact we can characterise $x$ as the only point at distance 0 from itself, looking at level-sets we get $\set{x}=\bigcap_{r>0}B(x,r)$.
We conclude with \begin{align*}
    \set{F_t(x)}&=\Part(F_t)(\set{x})=\Part(F_t)\Bigl(\bigcap_{r>0}B(x,r)\Bigr) \\
    &=\Part(F_t)\Bigl(\bigcap_{\delta(r)>0}B(x,r)\Bigr)
    \subset \bigcap_{r>0}B(x,r)=\set{x}.
\end{align*}
This is true regardless of the choice of $t\in T$, therefore the future is \[\set{F_t(x)\mid t\in T}=\bigcup\Part(F_t)(\set{x})=\bigcup \set{x}=\set{x}.\qedhere\] 
\end{proof}

\begin{theorem}\label{thm:Lyapunov}
    A Lyapunov function induces a Lyapunov equilibrium.
\end{theorem}
\begin{proof}
Using categorical cut and paste we get
\begin{equation}\label{eq:gluing}
    \begin{tikzcd}
	{\R_{>0}} & {\R_{>0}} & {\R_{>0}} \\
	{\Part(M)} && {\Part(M)}
	\arrow["{\Part(F_t)}"', from=2-1, to=2-3]
	\arrow[""{name=0, anchor=center, inner sep=0}, "{V^\leq}"', from=1-2, to=2-1]
	\arrow[""{name=1, anchor=center, inner sep=0}, "{V^\leq}", from=1-2, to=2-3]
	\arrow["A"', from=1-2, to=1-1]
	\arrow[""{name=2, anchor=center, inner sep=0}, "{B(x,\cdot)}"', from=1-1, to=2-1]
	\arrow[""{name=3, anchor=center, inner sep=0}, "{B(x,\cdot)}", from=1-3, to=2-3]
	\arrow["{B^{-1}}"', from=1-3, to=1-2]
	\arrow[shorten <=8pt, shorten >=8pt, Rightarrow, from=0, to=1]
	\arrow[shorten <=4pt, shorten >=4pt, Rightarrow, from=1, to=3]
	\arrow[shorten <=4pt, shorten >=4pt, Rightarrow, from=2, to=0]
\end{tikzcd}
\end{equation}
The left and center triangles are given by the definition of Lyapunov function. The upper right triangle can be obtained through Lemma \ref{lem:betainverse}. The outer square is the definition of Lyapunov equilibrium.
\end{proof}

We needed the following lemma to complete the proof. 
\begin{lemma}[Folk]\label{lem:inversetrianglecat} 
Let $F\colon A\to B,G\colon A\to C,H\colon B\to C$ be functors where $F$ is invertible. Then we have that
$$\begin{tikzcd}
	A \\
	& B \\
	C
	\arrow[""{name=0, anchor=center, inner sep=0}, "G"', from=1-1, to=3-1]
	\arrow["F", from=1-1, to=2-2]
	\arrow["H", from=2-2, to=3-1]
	\arrow[shorten <=5pt, Rightarrow, from=0, to=2-2]
\end{tikzcd}\iff \begin{tikzcd}
	A \\
	& B \\
	C
	\arrow[""{name=0, anchor=center, inner sep=0}, "G"', from=1-1, to=3-1]
	\arrow["F^{-1}"', to=1-1, from=2-2]
	\arrow["H", from=2-2, to=3-1]
	\arrow[shorten <=5pt, Rightarrow, from=0, to=2-2]
\end{tikzcd} $$
\end{lemma}
\begin{proof}

By composition of natural transformations we have
\[\begin{tikzcd}
	B \\
	A \\
    C
	\arrow[""{name=0, anchor=center, inner sep=0}, "F^{-1}",bend left, from=1-1, to=2-1]
	\arrow[""{name=1, anchor=center, inner sep=0}, "F^{-1}"', bend right, from=1-1, to=2-1]
	\arrow[""{name=2, anchor=center, inner sep=0}, "FH",bend left, from=2-1, to=3-1]
	\arrow[""{name=3, anchor=center, inner sep=0}, "G"',bend right, from=2-1, to=3-1]
	\arrow[""', shorten <=2pt, shorten >=2pt, Rightarrow, from=1, to=0]
	\arrow[""', shorten <=2pt, shorten >=2pt, Rightarrow, from=3, to=2]
\end{tikzcd} \] which implies there is a natural transformation from $F^{-1}G$ to $F^{-1}FH=\Id_AH=H$.
The other direction is similar.
\end{proof}
In our case, this translates to the following corollary.
\begin{corollary}\label{lem:betainverse} 
$$\begin{tikzcd}
	 {\R_{>0}} \\
	& {\R_{>0}} \\
	{\Part(M)}
	\arrow[""{name=0, anchor=center, inner sep=0}, "{V^\leq}"{description}, from=1-1, to=3-1]
	\arrow["B", from=1-1, to=2-2]
	\arrow["{B(x^*,\cdot)}", from=2-2, to=3-1]
	\arrow[ shorten <=6pt, Rightarrow, from=0, to=2-2]
\end{tikzcd}\iff \begin{tikzcd}
	 {\R_{>0}} \\
	& {\R_{>0}} \\
	{\Part(M)}
	\arrow[""{name=0, anchor=center, inner sep=0}, "{V^\leq}"{description}, from=1-1, to=3-1]
	\arrow["B^{-1}"', to=1-1, from=2-2]
	\arrow["{B(x^*,\cdot)}", from=2-2, to=3-1]
	\arrow[ shorten <=6pt, Rightarrow, from=0, to=2-2]
\end{tikzcd}$$
\end{corollary}

It is noteworthy that neither Lemma \ref{lem:inversetrianglecat} nor equation \eqref{eq:gluing} make any real use of the level-set. The level-set is hidden in two things: in the definition of Lyapunov stability and in the inclusion \eqref{eq:inclusion}. This inclusion is possible because we are restricting the domain of the morphism $V^{\leq}$, which is much easier than restricting the codomain of $V$. This asymmetry between restricting the domain and restricting the codomain is the main reason we look at level-set.

\section{The converse Lyapunov theorem}

The construction of the last diagram in the proof of \ref{thm:Lyapunov} turns out to be useful, in the next lemma, we show it can be obtained without loss of generality.

\begin{lemma}\label{lemma:step1}
The following conditions are equivalent:
\begin{tfae}
    \item There exist $\delta$ such that the following is satisfied \[\begin{tikzcd}[ampersand replacement=\&, column sep=small]
	{\R_{>0}} \& {\R_{>0}} \\
	{\Part(M)} \& {\Part(M)}
	\arrow["{\delta(\cdot)}", to=1-1, from=1-2,swap]
	\arrow[""{name=0, anchor=center, inner sep=0},"{B(x^*,\cdot)}"', from=1-1, to=2-1]
	\arrow["{\Part (F_{t_1})}"', from=2-1, to=2-2]
	\arrow[""{name=1, anchor=center, inner sep=0},"{B(x^*,\cdot)}", from=1-2, to=2-2]
	\arrow["\subset"{description} shorten >=3pt, Rightarrow, from=0, to=1]
\end{tikzcd}\]
    \item There exist $\delta_+,\delta_-,S_1,S_2$ such that the following is satisfied \[\begin{tikzcd}
	{\R_{>0}} & {\R_{>0}} & {\R_{>0}} \\
	{\Part(M)} && {\Part(M)}
	\arrow["{\delta_-}"', from=1-2, to=1-1]
	\arrow["{\delta_+}"', from=1-3, to=1-2]
	\arrow["{B(x^*,\cdot)}"', from=1-1, to=2-1]
	\arrow["{\Part(F_{t_1})}"', from=2-1, to=2-3]
	\arrow[""{name=0, anchor=center, inner sep=0}, "{B(x^*,\cdot)}", from=1-3, to=2-3]
	\arrow[""{name=1, anchor=center, inner sep=0}, "{S_1}"{description}, from=1-2, to=2-1]
	\arrow[""{name=2, anchor=center, inner sep=0}, "{S_2}"{description}, from=1-2, to=2-3]
	\arrow["\beta"', shorten >=3pt, Rightarrow, from=1-1, to=1]
	\arrow["\alpha"{description}, shorten >=8pt, Rightarrow, from=2-1, to=2]
	\arrow["\gamma", shorten <=7pt, Rightarrow, from=1-2, to=0]
\end{tikzcd}\]
\end{tfae}
\end{lemma}
\begin{proof}
(i) $\Rightarrow$ (ii) Set $\delta_+=\Id_{\R_{>0}}$, $\delta_-=\delta(\cdot)$ and $S_1=\delta_-\circ B(x^*,\cdot)=B(x^*,\delta_-)$, $S_2=B(x^*,\cdot)$. The left and right triangle commute by definition, so they admit the identity as natural transformation.
The inner triangle has a natural transformation given by hypothesis: $\alpha=\subset$.

(ii) $\Rightarrow$ (i) Set $\delta(\cdot)=\delta_+\circ\delta_-$ and look only at the outer square. Since all of the inner triangles admit a natural transformation, the outer square also admits a natural transformation, by composition.
\end{proof}

This bring us to our main theorem. Note that Lemma  \ref{lem:inversetrianglecat} does not depend on any property of time. However for the next step we will need to use the ordering of time and the fact the composition of time intervals is functorial.
\begin{theorem}\label{thm:Lyapunovconv}
 A forward dynamical system $F\colon T\to M$ has a Lyapunov equilibrium $x^*$ if and only if there is a Lyapunov function for it.
\end{theorem}
\begin{proof}
Theorem \ref{thm:Lyapunov} mean we only have to prove the necessity of having a Lyapunov function.
By Lemma  \ref{lemma:step1}, having a Lyapunov equilibrium means that we have \[\begin{tikzcd}[ampersand replacement=\&, column sep=small]
	{\R_{>0}} \& {\R_{>0}} \& {\R_{>0}} \\
	{\Part(M)} \&\& {\Part(M)}
	\arrow["{\delta_-}"', from=1-2, to=1-1]
	\arrow["{\delta_+}"', from=1-3, to=1-2]
	\arrow["{B(x^*,\cdot)}"', from=1-1, to=2-1]
	\arrow["{\Part(F_t)}"', from=2-1, to=2-3]
	\arrow[""{name=0, anchor=center, inner sep=0}, "{B(x^*,\cdot)}", from=1-3, to=2-3]
	\arrow[""{name=1, anchor=center, inner sep=0}, "{S_1}"{description}, from=1-2, to=2-1]
	\arrow[""{name=2, anchor=center, inner sep=0}, "{S_2}"{description}, from=1-2, to=2-3]
	\arrow[, shorten >=3pt, Rightarrow, from=1-1, to=1]
	\arrow["\subset"{description}, shorten >=8pt, Rightarrow, from=2-1, to=2]
	\arrow[, shorten <=7pt, Rightarrow, from=1-2, to=0]
\end{tikzcd}\]
with $\delta_+=\Id$ invertible. It remains to show we can construct $V^{\leq}$ respecting $\begin{tikzcd}[ampersand replacement=\&, column sep=small]
	{\Part(M)} \& \& {\Part(M)} \\
	\& {\R_{>0}}
	\arrow[""{name=0, anchor=center, inner sep=0},"V^\leq"', to=1-1, from=2-2,swap]
	\arrow["V^\leq", to=1-3, from=2-2,swap]
	\arrow["{\Part (F_t)}", to=1-1, from=1-3,swap]
	\arrow["\supset"{description}, shorten >=8pt, Rightarrow, from=1-3, to=0]
\end{tikzcd}$ and  $\begin{tikzcd}[ampersand replacement=\&, column sep=small]
	\& {\R_{>0}} \\
	{\R_{>0}} \&\& {\R_{>0}} \\
	\& {\Part(M)}
	\arrow[""{name=0, anchor=center, inner sep=0}, "{V^\leq}"{description}, from=1-2, to=3-2]
	\arrow["A"', from=1-2, to=2-1]
	\arrow["{B(x^*,\cdot)}"', from=2-1, to=3-2]
	\arrow["B", from=1-2, to=2-3]
	\arrow["{B(x^*,\cdot)}", from=2-3, to=3-2]
	\arrow[ shorten >=6pt, Rightarrow, from=2-1, to=0]
	\arrow[ shorten <=6pt, Rightarrow, from=0, to=2-3]
\end{tikzcd}$ with $B$ invertible.
If $S_1=S_2$ the proof ends there. Otherwise we construct explicitly $V^{\leq}$ by putting $$V^\leq(\epsilon)=\bigcup_{t_i\leq 0}F_{t_i}S_1(\epsilon)=\set{x\in M\mid \exists t\in T,\exists y\in S_1(\epsilon) : F_t(y)=x}.$$

One can check using set-theoretical methods that $S_1(\epsilon) \subset V^\leq(\epsilon)\subset S_2(\epsilon)$, which gives the diamond condition.
\[\begin{tikzcd}
	& {\R_{>0}} \\
	{\R_{>0}} && {\R_{>0}} \\
	& {\Part(M)}
	\arrow[""{name=0, anchor=center, inner sep=0}, "{V^\leq}"{description}, from=1-2, to=3-2]
	\arrow[""{name=1, anchor=center, inner sep=0}, "{S_2}"{description}, curve={height=-24pt}, from=1-2, to=3-2]
	\arrow[""{name=2, anchor=center, inner sep=0}, "{S_1}"{description}, curve={height=24pt}, from=1-2, to=3-2]
	\arrow["{\delta_+}"', from=1-2, to=2-1]
	\arrow["{B(x^*,\cdot)}"', from=2-1, to=3-2]
	\arrow["{\delta_-^{-1}}", from=1-2, to=2-3]
	\arrow["{B(x^*,\cdot)}", from=2-3, to=3-2]
	\arrow[shorten >=4pt, Rightarrow, from=2-1, to=2]
	\arrow[shorten <=4pt, Rightarrow, from=1, to=2-3]
	\arrow[shorten <=5pt, shorten >=4pt, Rightarrow, from=0, to=1]
	\arrow[shorten <=4pt, shorten >=5pt, Rightarrow, from=2, to=0]
\end{tikzcd}\]
Then we can compute
\begin{align*}
    F_t(V^\leq (\epsilon))&=F_{t}(\bigcup_{  t_i\geq 0}F_{t_i}(S_1(\epsilon)))
    =\bigcup_{ t_i\geq 0}F_{t+t_i}(S_1(\epsilon))\\
    &= \bigcup_{t_i'\geq t}F_{t_i'}(S_1(\epsilon))
    \subset \bigcup_{ t'\geq 0}F_{t'}(S_1(\epsilon))=V^\leq (\epsilon),
\end{align*}
which proves that $F$ decreases on trajectories.
\end{proof}
We will note that extra properties of $A$, $B$ can often be directly translated into extra properties of $\delta_-$, $\delta_+$, allowing this theorem to hold also for other notions of equilibrium (such as asymptotic and/or global equilibrium) provided we ask the corresponding restriction on Lyapunov functions. 
As an example, we say that the equilibrium is global if $\delta_-$ is invertible and the corresponding notion of Lyapunov function is to ask for $\alpha$ to be invertible.
\section{Strong, weak and other equilibria}
Now that we have established an equivalence for one notion of equilibrium, let us expand it to other notions of equilibrium that can be found in the literature.

We make a detour to note the following:
\begin{itemize}
    \item if $\delta$ is a proof for Lyapunov equilibrium, then the same is true for any $\delta'\leq \delta$;
    \item for any monotonic function $\delta \colon \R^+\to \R^+$, there is a $C^\infty$ monotonic function $\delta'$ such that $\delta'\leq \delta$;
    \item if $\delta$ is a continuous proof for a Lyapunov equilibrium, then it is a contraction $\R^+\to (\sup\delta,0)$;
    \item having a Lyapunov equilibrium with $\delta$ invertible is equivalent to having a Lyapunov equilibrium with $\delta$ such that $\sup Im (\delta) =+\infty$.
\end{itemize}
Balls of radius $\delta$ represent the region of attraction of the equilibrium. If the function $\delta(\cdot)$ is unbounded, it means every point of the state-space is attracted to the equilibrium, so we say it is a global equilibrium.
\begin{definition}
    A global equilibrium is an equilibrium with extra properties: we require from the diagram \[\begin{tikzcd}[ampersand replacement=\&, column sep=small]
	{\R_{>0}} \& {\R_{>0}} \\
	{\Part(M)} \& {\Part(M)}
	\arrow["{\delta(\cdot)}", to=1-1, from=1-2,swap]
	\arrow[""{name=0, anchor=center, inner sep=0},"{B(x^*,\cdot)}"', from=1-1, to=2-1]
	\arrow["{\Part (F_{t_1})}"', from=2-1, to=2-2]
	\arrow[""{name=1, anchor=center, inner sep=0},"{B(x^*,\cdot)}", from=1-2, to=2-2]
	\arrow["\subset"{description} shorten >=3pt, Rightarrow, from=0, to=1]
\end{tikzcd}\] that $\delta$ is invertible.
\end{definition}
\begin{definition}
    A global Lyapunov function is a Lyapunov function with the extra requirement that $A$ be invertible in the diagram of rough approximation of distance.
    \[\begin{tikzcd}
	& M \\
	\R^+ && \R^+ \\
	& \R^+
	\arrow["{d(x,\cdot)}"', from=1-2, to=2-1]
	\arrow[""{name=0, anchor=center, inner sep=0}, "V"{description}, from=1-2, to=3-2]
	\arrow["{d(x,\cdot)}", from=1-2, to=2-3]
	\arrow["A"', from=2-1, to=3-2]
	\arrow["B", from=2-3, to=3-2]
	\arrow[shorten >=5pt, Rightarrow, from=2-1, to=0]
	\arrow[shorten <=5pt, Rightarrow, from=0, to=2-3]
\end{tikzcd}\]
\end{definition}
\begin{remark}
    Thanks to Lemma \ref{lem:inversetrianglecat} this means that $V$ and $d(x^*)$ are both rough approximation of each other.
\end{remark}
\begin{theorem}
    A forward dynamical system has a global equilibrium if and only if it has a global Lyapunov function for it.
\end{theorem}
\begin{proof}
    Follow the proofs of Theorem \ref{thm:Lyapunov} and Theorem \ref{thm:Lyapunovconv}, making sure to note that $\delta$ being invertible implies in succession that $\delta_+$, $\alpha$ and $A$ are invertible. Same in the other direction.
\end{proof}

In fact the same definitions and proofs would stand for any property of $A$ and suitable properties assigned to $B$.

\begin{definition}
    Let $P$ be a set of functions, or a collection of morphisms. We say $f$ \textbf{has property} $P$ if $f\in P$.
\end{definition}
\begin{definition}
    Let $P_1$, $P_2$ be properties. We define the property $P_1\circ P_2$ as the property of every function $f_1\circ f_2$ with $f_i\in P_i$, $i=1,2$.
\end{definition}
\begin{definition}
    If $P$ is a property of invertible functions or morphisms, we define $P^{-1}$ as the property of every function $f^{-1}$ with $f\in P$.
\end{definition}

\begin{exemples}
    \begin{itemize}
        \item $P_{\text{cont}}$ is the set of all continuous functions; we say f is continuous if $f\in P_{\text{cont}}$.
        \item $P_{\text{aut}}$ is the set of automorphisms. One can check that $f\in P_{\text{aut}}\implies f\in P_{\text{aut}}^{-1}$.
        \item $P_{id}$ is the identity; one can check that for any other property $P_2$, we have $P_2=P_{id}\circ P_2 = P_2\circ P_{id}$.
    \end{itemize}
\end{exemples}

\begin{theorem}
    Let $F\colon T\to M$ be a forward dynamical system. The following are equivalent:
    \begin{tfae}
        \item the system admits an equilibrium where $\delta$ has property $P$;
        \item the system admits a Lyapunov function where $A$ has property $P$, and $B=\Id$.
    \end{tfae}
Furthermore, if the system admits a Lyapunov function where $A$ has property $P_1$ and $B$ has property $P_2$, then the system admits an equilibrium where $\delta$ has property $P_1\circ (P_2)^{-1}$.
\end{theorem}
\begin{proof}
    See the proof of Theorem \ref{thm:Lyapunov}.
\end{proof}

\section{Conclusion}
We have given a definition of level-set morphism in a set-theoretical context, which is easily seen to be valid in an arbitrary topos: indeed,  we never used the excluded middle in our proofs. We explained how a morphism relates to its level-set in a way similar to an adjunction. We have then given a categorical definition of stability and Lyapunov functions. We then showed that one is in fact just a reformulation of the other, and we explained how a change in the equilibrium translates to a corresponding change in the associated Lyapunov function. This transfer is the easy half of a two-way process. We hope in  further work to describe how changes in a Lyapunov function, such as requiring $V$ to be continuous, affect the associated equilibrium, or when to expect such behaviour for Lyapunov functions.

\section{acknowledgement}
RJ is a FNRS honorary Research Associate. This project has received funding from the European Research Council (ERC) under the European Union's Horizon 2020 research and innovation programme under grant agreement No 864017 - L2C. RJ is also supported by the Innoviris Foundation and the FNRS (Chist-Era Druid-net).
\bibliographystyle{plain}
\bibliography{bibli}
\end{document}